\documentclass{article}
\usepackage{amsmath}
\usepackage{amssymb}

\newtheorem{theorem}{Theorem}

\newtheorem{definition}[theorem]{Definition}

\newtheorem{proposition}[theorem]{Proposition}
\newtheorem{remark}[theorem]{Remark}

\newenvironment{proof}[1][Proof]{\noindent\textbf{#1.} }{\ \rule{0.5em}{0.5em}}
\input xypic
\input xy
\xyoption{all}
\begin{document}

\author{ U. E. Arslan, Z. Arvasi and G. Onarl\i\ }
\title{\textbf{(Co)-Induced Two-Crossed Modules}}
\date{}
\maketitle

\begin{abstract}
We introduce the notion of an (co)-induced $2$-crossed module, which
extends the notion of an (co)-induced crossed module (Brown and
Higgins).
\end{abstract}


\section*{Introduction}

Induced crossed modules were defined by Brown and Higgins \cite{BROWN and
HIGGINS} and studied further in paper by Brown and Wensley \cite{BROWN and
WENSLEY,BROWN and WENSLEY2},\cite{BROWN1}. This is looked at in detail in a book by Brown,
Higgins and Sivera \cite{BHS}. Induced crossed modules allow detailed
computations of non-Abelian information on second relative groups.

To obtain analogous result in dimension $3$, we make essential use of a $2$%
-crossed module defined by Conduch\'{e} \cite{CONDUCHE}.

A major aim of this paper is given any $\theta =(\phi ^{\prime
},\phi ):(M,P,\alpha )\rightarrow (N,Q,\beta )$ pre-crossed module morphism to introduce induced $(N \rightarrow Q)$-$2$-crossed module%
\begin{equation*}
\left\{ \theta _{\ast }\left( L\right) ,N ,Q,\partial _{\ast},\beta
\right\}
\end{equation*}
which can be used in applications of the $3$-dimensional Van Kampen Theorem.

The method of Brown and Higgins \cite{BROWN and HIGGINS} is
generalized to give results on $\{ \theta _{\ast }\left( L\right) ,N
,Q,\partial _{\ast},\beta\}.$ Also we construct pullback
(co-induced) $2$-crossed modules in terms of the concept of pullback
pre-crossed modules in \cite{BHS}. However; Brown, Higgins and Sivera
\cite{BHS} indicate a bifibration from crossed squares, so leading
to the notion of induced crossed square, which is relevant to a
triadic Hurewicz theorem results in dimension $3$.\\

\footnotetext[1]{We are grateful to referee and D. Conduch\'{e} for helpful comments.}



\section{Preliminaries\hspace{1cm}}

Throughout this paper all actions will be left. The right actions in
some references will be rewritten by using left actions.

\subsection{Crossed Modules}

Crossed modules of groups were initially defined by Whitehead \cite{w1}
as models for (homotopy) $2$-types. We recall from \cite{PORTER} the
definition of crossed modules of groups.

A crossed module, $(M,P,\partial )$, consists of groups $M$ and $P$ with a
left action of $P$ on $M$, written $\left( p,m\right) \mapsto $ $^{p}m$ and
a group homomorphism $\partial :M\rightarrow P$ satisfying the following
conditions:%
\begin{equation*}
CM1)\ \partial \left( ^{p}m\right) =p\partial(
m)p^{-1}\text{\qquad\ and \qquad }CM2)\
^{\partial \left( m\right) }n=mnm^{-1}\text{ }
\end{equation*}%
for $p\in P,m,n\in M$. We say that $\partial :M\rightarrow P$ is a
pre-crossed module, if it satisfies CM$1.$

If $(M,P,\partial )$ and $(M^{\prime },P^{\prime },\partial ^{\prime })$ are
crossed modules, a morphism,
\begin{equation*}
\left( \mu ,\eta \right) :(M,P,\partial )\rightarrow (M^{\prime },P^{\prime
},\partial ^{\prime }),
\end{equation*}%
of crossed modules consists of group homomorphisms $\mu :M\rightarrow
M^{\prime }$ and\ $%
\begin{array}{c}
\eta :P\rightarrow P^{\prime }%
\end{array}%
$ such that%
\begin{equation*}
(i)\text{ }\eta \partial =\partial ^{\prime }\mu \text{ \qquad and \qquad }%
(ii)\text{ }\mu \left( ^{p}m\right) =\text{ }^{\eta (p)}\mu \left( m\right)
\text{ }
\end{equation*}%
for all $p\in P,m\in M.$

Crossed modules and their morphisms form a category, of course. It will
usually be denoted by \textsf{XMod}. We also get obviously a category
\textsf{PXMod} of pre-crossed modules.

There is, for a fixed group $P$, a subcategory \textsf{XMod}$/P$ of \textsf{%
XMod}, which has as objects those crossed modules with $P$ as the
\textquotedblleft base\textquotedblright , i.e., all $(M,P,\partial )$ for
this fixed $P$, and having as morphisms from $\left( M,P,\partial \right) $
to $(M^{\prime },P^{\prime },\partial ^{\prime })$ those $\left( \mu ,\eta
\right) $ in \textsf{XMod} in which $\eta :P\rightarrow P^{\prime }$ is the
identity homomorphism on $P.$

Some standart examples of crossed modules are:

(i) normal subgroup crossed modules $\left( i:N\rightarrow P\right) $ where $%
i$ is an inclusion of a normal subgroup, and the action is given by
conjugation;

(ii) automorphism crossed modules $\left( \chi :M\rightarrow Aut(M)\right) $
in which
\begin{equation*}
\left( \chi m\right) \left( n\right) =mnm^{-1};
\end{equation*}

(iii) Abelian crossed modules $0:M\rightarrow P$ where $M$ is a $P$-module;

(iv) central extension crossed modules $\partial :M\rightarrow P$ where $%
\partial $ is an epimorphism with kernel contained in the center of $M.$

\subsection{Pullback Crossed Modules}\label{g1}
We recall from \cite{BHS} below a presentation of the pullback (co-induced) crossed module.

Let $\phi :P\rightarrow Q$ be a homomorphism of groups and let $%
\begin{array}{c}
\mathcal{N}=(N,Q,v)%
\end{array}%
$ be a crossed module. We define a subgroup
\begin{equation*}
\phi ^{\ast }(\mathcal{N})=N\times _{Q}P=\{(n,p)\mid v(n)=\phi (p)\}
\end{equation*}%
of the product $N\times P.$ This is the usual pullback in the
category of
groups. There is a commutative diagram%
\begin{equation*}
\xymatrix { \phi^*(N) \ar[r]^-{\bar{\phi}} \ar[d]_-{\bar{v}} & N \ar[d]^-v & \\
P \ar[r]_-\phi & Q & }
\end{equation*}%
where $\bar{v}:(n,p)\mapsto p,$ $\bar{\phi}:(n,p)\mapsto n.$ Then
$P$ acts on $\phi ^{\ast }(N)$ via $\phi $ and the diagonal, i.e.
$^{p^{\prime }}(n,p)=(^{\phi (p^{\prime })}n,p^{\prime }pp^{\prime
-1}).$ This gives a $P$-action. Since
\begin{equation*}
\begin{array}{ccl}
(n,p)(n_{1},p_{1})(n,p)^{-1} & = & \left( nn_{1}n^{-1},pp_{1}p^{-1}\right)
\\
& = & \left( ^{v\left( n\right) }n_{1},pp_{1}p^{-1}\right) \\
& = & \left( ^{\phi \left( p\right) }n_{1},pp_{1}p^{-1}\right) \\
& = & ^{\bar{v}(n,p)}(n_{1},p_{1}),%
\end{array}%
\end{equation*}%
we get a crossed module $\phi ^{\ast }(\mathcal{N})=(\phi ^{\ast }(N),P,\bar{%
v})$ which is called the $pullback$ $crossed$ $module$ of $\mathcal{N}$
along $\phi $. This construction satisfies a universal property, analogous
to that of the pullback of groups. To state it, we use also the morphism of
crossed modules%
\begin{equation*}
\left( \bar{\phi},\phi \right) :\phi ^{\ast }(\mathcal{N})\rightarrow
\mathcal{N}.
\end{equation*}

\begin{theorem}
For any crossed module $\mathcal{M}=(M,P,\mu )$ and any morphism of crossed
modules%
\begin{equation*}
(h,\phi ):\mathcal{M\rightarrow N},
\end{equation*}%
there is a unique morphism of crossed $P$-modules $h^{\prime }:\mathcal{M}%
\rightarrow \phi ^{\ast }(\mathcal{N})$ such that the following diagram
commutes%
\begin{equation*}
\xymatrix { M \ar@/^/[drr]^h \ar@/_/[ddr]_{\mu} \ar@{-->}[dr]^{h^{\prime }}
& & & \\ & \phi^*(N) \ar[r]_{\bar{\phi}} \ar[d]^{\bar{v}} & N \ar[d]_v & \\
& P \ar[r]^{\phi} & Q. & }
\end{equation*}%
\end{theorem}

This can be expressed functorially:

\begin{equation*}
\mathbf{\phi }^{\ast }\mathbf{:}\ \mathsf{XMod}\mathbf{/}Q\mathbf{\ \
\rightarrow }\ \mathsf{XMod}\mathbf{/}P
\end{equation*}%
is a pullback functor. This functor has a left adjoint%
\begin{equation*}
\mathbf{\phi }_{\ast }\mathbf{:}\ \mathsf{XMod}\mathbf{/}P\mathbf{\ \
\rightarrow }\ \mathsf{XMod}\mathbf{/}Q
\end{equation*}
which gives an induced crossed module as follows.

Induced crossed modules were defined by Brown and Higgins in \cite{BROWN and
HIGGINS} and studied further in papers by Brown and Wensley \cite{BROWN and
WENSLEY,BROWN and WENSLEY2}.

We recall from \cite{BHS} below a presentation of the induced crossed module
which is helpful for the calculation of colimits.

\subsection{Induced Crossed Modules}\label{go}

In this section we will briefly explain Brown and Higgins'
construction of induced (pre-)crossed
modules in \cite{BROWN and HIGGINS} to compare $3$-dimensional construction given in section
\ref{i2xm}.

\begin{definition}
For any crossed $P$-module $\mathcal{M}=(M,P,\mu )$ and any homomorphism $%
\phi :P\rightarrow Q$ the crossed module $induced$ by $\phi $ from $\mu $
should be given by:

1$(i)$ a crossed $Q$-module $\phi _{\ast }\left( \mathcal{M}\right) =(\phi
_{\ast }\left( M\right),Q,\phi _{\ast }\mu ),$

$(ii)$ a morphism of crossed modules $\left( f,\phi \right):\mathcal{M}%
\rightarrow \phi_{\ast }\left( \mathcal{M}\right),$ satisfying the dual
universal property that for any morphism of crossed modules%
\begin{equation*}
(h,\phi ):\mathcal{M\rightarrow N}
\end{equation*}%
there is a unique morphism of crossed $Q$-modules $h^{\prime }:\phi _{\ast
}(M)\rightarrow N$ such that the diagram%
\begin{equation*}
\xymatrix { & & N \ar@/^/[ddl]^-v & \\ M \ar[r]_-f \ar@/^/[urr]^-h \ar[d]_-\mu &
\phi _*(M) \ar[d]_-{\phi _*\mu} \ar@{-->}[ur]^-{h^\prime}& & \\ P \ar[r]_-\phi
& Q & & }
\end{equation*}%
commutes.
\end{definition}

The crossed module $\phi _{\ast }\left( \mathcal{M}\right) =(\phi
_{\ast }\left( M\right) $,$Q,\phi _{\ast }\mu ) $ is called the induced crossed module of $\mathcal{M}=(M,P,\mu )$ along $\phi$.

\subsubsection{Construction of Induced Crossed Modules} \label{cxm}

The free $Q$-group $^{Q}M$ generated by $M$ is the kernel of the canonical
morphism $M\ast Q\rightarrow Q$ where $M\ast Q$ is the free product. Hence
it is generated by elements $qmq^{-1}$ with $q\in Q$ and $m\in M.$ It is
equivalent to say that $^{Q}M$ is generated by the set $Q\times M$ with
relation
\begin{equation*}
(q,m_{1})(q,m_{2})=(q,m_{1}m_{2})\ \ \ \ (\ref{cxm}.1)
\end{equation*}%
for $m_{1},m_{2}\in M,q\in Q.$

Then $Q$ acts on $^{Q}M$ by
\begin{equation*}
^{q^{\prime }}(q,m)=(q^{\prime }q,m)
\end{equation*}%
for $m\in M$ and $q^{\prime },q\in Q.$

Let $\mu :M\rightarrow P$ be a crossed module and $\phi
:P\rightarrow Q$ be a morphism of groups. As $^{Q}M$ is the free
$Q$-group, we have the commutative diagram

\[
\xymatrix @R=30pt@C=35pt{ M \ar[r] \ar[d]_-\mu \ar[dr]_-{\phi\mu} & ^Q{M} \ar[d]^-{\mu^\prime} & \\
\ P \ar[r]_-\phi & Q. & }
\]
That is the morphism $\phi\mu $ extends to the morphism ${\mu^\prime
}:\ ^{Q}M\rightarrow Q$ given by

\[
\mu ^{\prime }(q,m)=q\phi \mu (m)q^{-1}
\]

\noindent
for $m\in M,q\in Q.$

Thus $\mu ^{\prime }$ is the free $Q$-pre-crossed module generated
by $M$. To get a free crossed module we have to divide $^{Q}M$ by
the Peiffer subgroup, that is by relation:
\begin{equation*}
(q_{1},m_{1})(q_{2},m_{2})(q_{1},m_{1})^{-1}=(q_{1}\phi \mu
(m_{1})q_{1}^{-1}q_{2},m_{2})\ \ \ \ (\ref{cxm}.2)
\end{equation*}%
for $m_{1},m_{2}\in M$ and $q_{1},q_{2}\in Q.$

Next step, to obtain the \textit{induced }pre-crossed module, we must
identify the action of $P$ on $M$ with the action of $\phi (P)$ on $^{Q}M,$ that is we must divide by relation%
\begin{equation*}
(q,^{p}m)=(q\phi (p),m)
\end{equation*}%
for $m\in M,p\in P$ and $q\in Q.$ Then we have
the induced crossed module denoted by $\phi _{\ast }(M)$, dividing by relation (\ref{cxm}.2).

Thus Brown and Higgins proved the following result in \cite{BROWN and HIGGINS}.
\begin{proposition} \label{go2}
Let $\mu :M\rightarrow P$ be a crossed $P$-module and let $\phi
:P\rightarrow Q$ be a morphism of groups. Then the induced crossed $Q$%
-module $\phi _{\ast }\left( M\right) $ is generated, as a group, by the set
$Q\times M$ with defining relations
\end{proposition}

$%
\begin{array}{cl}
(i) & (q,m_{1})(q,m_{2})=(q,m_{1}m_{2}), \\
(ii) & (q,^{p}m)=(q\phi (p),m), \\
(iii) & (q_{1},m_{1})(q_{2},m_{2})(q_{1},m_{1})^{-1}=(q_{1}\phi \mu
(m_{1})q_{1}^{-1}q_{2},m_{2})%
\end{array}%
$ \newline

\noindent for $m,m_{1},m_{2}\in M,$ $q,q_{1},q_{2}\in Q$ and $p\in P.$

The morphism $\phi _{\ast }\mu :\phi _{\ast }\left( M\right) \rightarrow Q$
is given by $\phi _{\ast }\mu (q,m)=q\phi \mu (m)q^{-1}$, the action of $Q$
on $\phi _{\ast }\left( M\right) $ by $^{q}(q_{1},m)=(qq_{1},m),$ and the
canonical morphism $\phi ^{\prime }:M\rightarrow \phi _{\ast }\left(
M\right) $ by $\phi ^{\prime }(m)=(1,m).$

If $\phi :P\rightarrow Q$ is an epimorphism,\ the induced crossed module $%
\left( \phi _{\ast }\left( M\right),Q,\phi _{\ast }\mu \right) $ has a
simpler description.

\begin{proposition} \label{gs1}
$\left( \cite{BROWN and HIGGINS}\text{, Proposition }9\right) $ If $\phi
:P\rightarrow Q$ is an epimorphism, and $\mu :M\rightarrow P$ is a crossed
module, then $\phi _{\ast }(M)\cong M/[K,M],$ where $K=$Ker$\phi $, and $%
[K,M]$ denotes the subgroup of $M$ generated by all\ $^{k}mm^{-1}$ for all $%
m\in M,k\in K.$
\end{proposition}

\begin{proposition}\label{gs2}
$\left( \cite{BROWN and HIGGINS}\text{, Proposition }10\right) $ If
$\phi :P\rightarrow Q$ is an injection and $\mu :M\rightarrow P$ is
a crossed module, let $T$ be a left transversal of $\phi (P)$ in
$Q,$ and let $B$ be the free product of groups $_{T}M$ $(t\in T)$
each isomorphic with $M$ by an isomorphism $m\mapsto $ $_{t}m$
$(m\in M).$ Let $q\in Q$ acts on $B$ by the rule $^{q}(_{t}m)=$
$_{u}(^{p}m)$ where $p\in P,u\in T,$ and $qt=u\phi (p).$ Let $\delta
:B\rightarrow Q$ be defined by $_{t}m\mapsto $ $t(\phi \mu
m)t^{-1}.$ Then $\mu _{\ast }(M)=B/S$ where $S$ is the normal
closure in $B$ of the elements $bcb^{-1}(^{\delta b}c^{-1})$ for
$b,c\in B.$
\end{proposition}

\section{Two-Crossed Modules}

Conduch\'{e} \cite{CONDUCHE} described the notion of a $2$-crossed module as
a model of connected homotopy $3$-types.

A $2$\textit{-crossed module} is a normal complex of groups $L\overset{%
\partial _{2}}{\rightarrow }M\overset{\partial _{1}}{\rightarrow
}P$,
that is $\partial _{2}(L)\unlhd M,\partial _{1}(M)\unlhd P$ and
$\partial _{1}\partial _{2}=1,$ together with an action of $P$ on
all three groups and a mapping
\begin{equation*}
\left\{ -,-\right\} :M\times M\rightarrow L
\end{equation*}%
which is often called the Peiffer lifting such that the action of $P$ on
itself is by conjugation, $\partial _{2}$ and $\partial _{1}$ are $P$%
-equivariant.

$%
\begin{array}{crcl}
\mathbf{PL1}: & \partial _{2}\left\{ m_{0},m_{1}\right\} & = &
m_{0}m_{1}m_{0}^{-1}\left( ^{\partial _{1}m_{0}}m_{1}^{-1}\right) \\
\mathbf{PL2}: & \left\{ \partial _{2}l_{0},\partial _{2}l_{1}\right\} & = &
\left[ l_{0},l_{1}\right] \\
\mathbf{PL3}: & \left\{ m_{0},m_{1}m_{2}\right\} & = &
^{m_{0}m_{1}m_{0}^{-1}}\left\{ m_{0},m_{2}\right\} \left\{
m_{0},m_{1}\right\} \\
& \left\{ m_{0}m_{1},m_{2}\right\} & = & \left\{
m_{0},m_{1}m_{2}m_{1}^{-1}\right\} \left( ^{\partial _{1}m_{0}}\left\{
m_{1},m_{2}\right\} \right) \\
\mathbf{PL4}: & a)\qquad \text{\qquad }\left\{ \partial _{2}l,m\right\} & =
& l\left( ^{m}l^{-1}\right) \\
& b)\qquad \qquad \left\{ m,\partial _{2}l\right\} & = & ^{m}l\left(
^{\partial _{1}m}l^{-1}\right) \\
& \text{or }\left\{ \partial _{2}l,m\right\} \left\{ m,\partial _{2}l\right\}
& = & l\left( ^{\partial _{1}m}l^{-1}\right) \\
\mathbf{PL5}: & ^{p}\left\{ m_{0},m_{1}\right\} & = & \left\{
^{p}m_{0},^{p}m_{1}\right\}%
\end{array}%
$

\noindent for all $m,m_{0},m_{1},m_{2}\in M,l,l_{0},l_{1}\in L$ and $p\in P.$\\
Note that we have not specified how $M$ acts on $L$. In \cite{CONDUCHE},
that as follows: if $m\in M$ and $l\in L$, define
\begin{equation*}
^{m}l=l\left\{ \partial _{2}l^{-1},m\right\} .
\end{equation*}%
From this equation $\left( L,M,\partial _{2}\right) $ becomes a crossed
module$.$

We denote such a $2$-crossed module of groups by $\left\{ L,M,P,\partial
_{2},\partial _{1}\right\} .$

A morphism of $2$-crossed modules is given by a diagram

\begin{equation*}
\xymatrix { L \ar[r]^{\partial_{2}} \ar[d]_{f_2} & M \ar[r]^{\partial_{1}}
\ar[d]_{f_1} & P \ar[d]_{f_0} & \\ \ L' \ar[r]_{\partial^{'}_{2}} & M'
\ar[r]_{\partial^{'}_{1}} & P' & }
\end{equation*}%
where $%
\begin{array}{c}
f_{0}\partial _{1}=\partial _{1}^{\prime }f_{1}%
\end{array}%
$,$%
\begin{array}{c}
f_{1}\partial _{2}=\partial _{2}^{\prime }f_{2}%
\end{array}%
$%
\begin{equation*}
\begin{array}{ccc}
f_{1}\left( ^{p}m\right) =\text{ }^{f_{0}\left( p\right) }f_{1}\left(
m\right) & , & f_{2}\left( ^{p}l\right) =\text{ }^{f_{0}\left( p\right)
}f_{2}\left( l\right)%
\end{array}%
\end{equation*}%
and
\begin{equation*}
\left\{ -,-\right\} \left( f_{1}\times f_{1}\right) =f_{2}\left\{ -,-\right\}
\end{equation*}

\noindent \ for all $m\in M,l\in L$ and $p\in P.$

These compose in an obvious way giving a category which we will
denote by \textsf{X}$_{2}$\textsf{Mod}$.$ There is, for a fixed
group $P$, a subcategory \textsf{X}$_{2}$\textsf{Mod}${/P}$ of
\textsf{X}$_{2}$\textsf{Mod} which has as objects those crossed
modules with $P$ as the \textquotedblleft base\textquotedblright ,
i.e., all $\left\{ L,M,P,\partial _{2},\partial _{1}\right\} $ for
this fixed $P$, and having as morphism from $\left\{ L,M,P,\partial
_{2},\partial _{1}\right\} $ to $\left\{ L^{\prime },M^{\prime
},P^{\prime },\partial _{2}^{\prime },\partial _{1}^{\prime
}\right\} $ those $\left( f_{2},f_{1},f_{0}\right) $ in \textsf{X}$_{2}$%
\textsf{Mod} in which $f_{0}:P\rightarrow P^{\prime }$ is the
identity homomorphism on $P.$ Similarly we get a subcategory
\textsf{X}$_{2}$\textsf{Mod}${/(M,P)}$ of
\textsf{X}$_{2}$\textsf{Mod} for a fixed pre-crossed module
$M\rightarrow P$.

Some remarks on trivial Peiffer lifting of $2$-crossed modules given by Porter in
\cite{PORTER} are:

\noindent Suppose we have a $2$-crossed module%
\begin{equation*}
L\overset{\partial _{2}}{\rightarrow }M\overset{\partial _{1}}{\rightarrow }%
P,
\end{equation*}%
with extra condition that $\left\{ m,m^{\prime }\right\} =1$ for all $%
m,m^{\prime }\in M.$ The obvious thing to do is to see what each of the
defining properties of a $2$-crossed module give in this case.

(i) There is an action of $P$ on $L$ and$\ M$ and the $\partial $s are $P$%
-equivariant. (This gives nothing new in our special case.)

(ii) $\left\{ -,-\right\} $ is a lifting of the Peiffer commutator so if $%
\left\{ m,m^{\prime }\right\} =1,$ the Peiffer identity holds for $\left(
M,P,\partial _{1}\right) ,$ i.e. that is a crossed module;

(iii) if $l,l^{\prime }\in L,$ then $1=\left\{ \partial _{2}l,\partial
_{2}l^{\prime }\right\} =\left[ l,l^{\prime }\right] ,$ so $L$ is Abelian

\noindent and,

(iv) as $\left\{ -,-\right\} $ is trivial $^{\partial _{1}m}l^{-1}=l^{-1},$
so $\partial M$ has trivial action on $L.$

\noindent Axioms PL$3$ and PL$5$ vanish.

\textbf{Examples of }$\mathbf{2}$\textbf{-Crossed Modules}

\noindent $1.$\ Let $M\overset{{\partial
_{1} }}{\rightarrow }P$ be a pre-crossed
module$.$ Consider the Peiffer subgroup $\langle M,M\rangle \subset M$,
generated by the Peiffer commutators%
\begin{equation*}
\langle m,m^{\prime }\rangle =mm^{\prime }m^{-1}(^{\partial
_{1}(m)}m^{\prime -1})
\end{equation*}

\noindent for all $m,m^{\prime }\in M.$ Then
\begin{equation*}
\langle M,M\rangle \overset{\partial _{2}}{\rightarrow }M\overset{\partial
_{1}}{\rightarrow }P
\end{equation*}%
is a $2$-crossed module with the Peiffer lifting $\left\{ m,m^{\prime
}\right\} =\langle m,m^{\prime }\rangle ,$ \cite{Martins}.

\noindent $2.$\ Any crossed module gives a $2$-crossed module. If $\left(
M,P,\partial \right) $\textbf{\ }is a crossed module, the resulting sequence
\begin{equation*}
L\rightarrow M\rightarrow P
\end{equation*}%
is a $2$-crossed module by taking $L=1.$ This is functorial and \textsf{XMod}
can be considered to be a full category of \textsf{X}$_{2}$\textsf{Mod} in this way.
It is a reflective subcategory since there is a reflection functor obtained
as follows:

If $L\overset{\partial _{2}}{\rightarrow }M\overset{\partial _{1}}{%
\rightarrow }P$ is a $2$-crossed module, then Im$\partial _{2}$ is a normal
subgroup of $M$ and there is an induced crossed module structure on $%
\partial _{1}:\dfrac{M}{\text{Im}\partial _{2}}\rightarrow P,$ $\left( \text{%
c.f. }\cite{PORTER}\right) $.

\begin{remark}
$1.$ Another way of encoding $3$-types is using the noting of a
crossed square by\ Guin-Wal\'{e}ry and Loday, \cite{wl} . \\

\noindent $2.$ A crossed square can be considered as a complex of crossed
modules of length one and
thus, Conduch\'{e}, gave a direct proof from crossed squares to $%
2$-crossed modules. For this construction see \cite{cond}.
\end{remark}

\section{Pullback Two-Crossed Modules}

In this section we introduce the notion of a pullback $2$-crossed module,
which extends a pullback crossed module defined by Brown-Higgins, $\cite%
{BROWN and HIGGINS}$. The importance of the \textquotedblleft
pullback\textquotedblright\ is that it enables us to move from crossed $Q$%
-module to crossed $P$-module,\ when a morphism of groups $\phi
:P\rightarrow Q$ is given.

\begin{definition}
Given a $2$-crossed module $\left\{ H,N,Q,\partial _{2},\partial
_{1}\right\} $ and a morphism of groups $\phi :P\rightarrow Q,$ the
pullback $2$-crossed module is given by

(i) a $2$-crossed module $\phi ^{\ast }\left\{ H,N,Q,\partial _{2},\partial
_{1}\right\} =\left\{ \phi ^{\ast }(H),\phi ^{\ast }(N),P,\partial
_{2}^{\ast },\partial _{1}^{\ast }\right\} $

(ii) given any morphism of $2$-crossed modules%
\begin{equation*}
\left( f_{2},f_{1},\phi \right) :\left\{ B_{2},B_{1},P,\partial _{2}^{\prime
},\partial _{1}^{\prime }\right\} \rightarrow \left\{ H,N,Q,\partial
_{2},\partial _{1}\right\} ,
\end{equation*}%
there is a unique $\left( f_{2}^{\ast },f_{1}^{\ast },id_{P}\right) $ $2$%
-crossed module morphism that makes the following diagram commute:
\begin{equation*}
\xymatrix { & & (B_2,
B_1,P,{\partial_2}^{\prime},{\partial_1}^{\prime})
\ar@{-->}[dll]_{(f_2^*,f_1^*,id_P)} \ar[d]^{(f_2,f_1,\phi)} & \\ \
(\phi^*(H),\phi^*(N), P,
\partial_{2}^{*},\partial_{1}^{*})\ar[rr]_-{(\phi^{\prime \prime},\phi^{\prime},\phi)} & & (H, N,
Q, \partial_{2},\partial_{1})& }
\end{equation*}%
or more simply as

\[
\xymatrix @R=20pt@C=25pt {  B_2  \ar@{-->}[dr]_{f_2^\ast} \ar@/^/[drrr]^{f_2} \ar[dd]_{{\partial_2}^{\prime}} & & & & \\
\ & {\phi^\ast}(H) \ar[dd]_{\partial_2^\ast} \ar[rr]_{\phi^{\prime \prime}}& & H \ar[dd]^{\partial_2}& \\
\ B_1  \ar@{-->}[dr]_{f_1^\ast} \ar@/^/[drrr]^{f_1} \ar[dd]_{{\partial_1}^{\prime}} & & & & \\
\ & {\phi^*}(N) \ar[dd]_{\partial_1^\ast} \ar[rr]_{\phi^{\prime}}& & N \ar[dd]^{\partial_1}& \\
\ P \ar@{=}[dr]_{id_P} \ar@/^/[drrr]^\phi  & & & & \\
\ & P \ar[rr]_\phi & & Q . & }
\]
\end{definition}
\subsubsection{Construction of Pullback $2$-Crossed Modules} \label{cxm1}

We can construct pullback $2$-crossed modules by using the notion of the pullback pre-crossed module in \cite{BHS}. Let
\begin{equation*}
(f_{2},f_{1},\phi ):\left\{ B_{2},B_{1},P,\partial _{2}^{\prime
},\partial _{1}^{\prime }\right\} \rightarrow \left\{ H,N,Q,\partial
_{2},\partial _{1}\right\}
\end{equation*}%
be a $2$-crossed module morphism and we consider $(\phi^{\ast}(N),P,\partial_1^{\ast})$ the pullback pre-crossed
module of $N\rightarrow Q$ by $\phi $ given in subsection
\ref{g1}. Since $\partial _{1}\partial
_{2}f_{2}=\phi \partial _{1}^{\prime }\partial _{2}^{\prime }=1,$ the image $%
f_{2}(B_{2})$ must be contained in $\partial _{2}^{-1}(Ker\partial
_{1}).$ Whence the pullback $2$-crossed module associated to $\phi $
is
\[
\partial _{2}^{-1}(Ker\partial _{1})\overset{\partial
_{2}^{\ast }}{\rightarrow }\phi ^{\ast }(N)\overset{\partial
_{1}^{\ast }}{\rightarrow }P
\]
where $\partial _{2}^{\ast
}(h)=(\partial _{2}(h),1)$\ for $h \in \partial _{2}^{-1}(Ker\partial _{1})$.

\begin{theorem}
\label{ug3}

If $H\overset{\partial _{2}}{\rightarrow }N\overset{\partial _{1}}{%
\rightarrow }Q$ is a $2$-crossed module and $\phi :P\rightarrow Q$
is a morphism of groups then
\begin{equation*}
\partial _{2}^{-1}(Ker\partial _{1})\overset{\partial _{2}^{\ast }}{%
\rightarrow }\phi ^{\ast }(N)\overset{\partial _{1}^{\ast
}}{\rightarrow }P
\end{equation*}%
is a pullback $\ 2$-crossed module where
\begin{equation*}
\begin{array}{ccc}
\partial _{2}^{\ast
}(h)=(\partial _{2}(h),1)&  &\partial _{1}^{\ast }(n,p)=p,
\end{array}%
\end{equation*}
the action of $P$ on $\phi
^{\ast }\left( N\right)$ and $\partial _{2}^{-1}(Ker\partial _{1})$ by $%
^{p}(n,p^{\prime })=\left( ^{\phi \left( p\right) }n,pp^{\prime
}p^{-1}\right) $ and $^{p}h=\ ^{\phi (p)}h,$ respectively.
\end{theorem}

\begin{proof} (i) Since $\partial _{1}^{\ast }\partial _{2}^{\ast }(h)=\partial _{1}^{\ast
}(\partial _{2}(h),1),$
\[
\partial _{2}^{-1}(Ker\partial _{1})\overset{%
\partial _{2}^{\ast }}{\rightarrow }\phi ^{\ast }(N)\overset{\partial
_{1}^{\ast }}{\rightarrow }P \]
is a normal complex of groups.
$\partial _{2}^{\ast }$ is $P$-equivariant with the action
$^{p}(n,p^{\prime })=\left( ^{\phi \left( p\right) }n,pp^{\prime
}p^{-1}\right).$
\begin{equation*}
\begin{array}{rcl}
^{p}\partial _{2}^{\ast }\left( h\right) & = & ^{p}(\partial
_{2}\left(
h\right) ,1) \\
& = & (^{\phi \left( p\right) }\partial _{2}\left( h\right) ,p1p^{-1}) \\
& = & (^{\phi \left( p\right) }\partial _{2}\left( h\right) ,1) \\
& = & (\partial _{2}\left( ^{\phi \left( p\right) }h\right) ,1) \\
& = & (\partial _{2}\left( ^{p}h\right) ,1) \\
& = & \partial _{2}^{\ast }\left( ^{p}h\right)%
\end{array}%
\end{equation*}
It is clear that $\partial _{1}^{\ast }$ is $P$-equivariant. The
Peiffer
lifting%
\begin{equation*}
\left\{ -,-\right\} :\phi ^{\ast }(N)\times \phi ^{\ast
}(N)\rightarrow \partial _{2}^{-1}(Ker\partial _{1})
\end{equation*}%
is given by $\left\{ \left( n,p\right) ,\left( n^{\prime },p^{\prime
}\right) \right\} =\{n,n^{\prime }\}.$

\textbf{PL1:}%
\begin{equation*}
\begin{array}{ll}
& \left( n,p\right) \left( n^{\prime },p^{\prime }\right) \left(
n,p\right) ^{-1}\left( ^{\partial _{1}^{\ast }\left( n,p\right)
}\left( n^{\prime
},p^{\prime }\right) ^{-1}\right) \\
= & \left( n,p\right)\left( n^{\prime },p^{\prime }\right) \left(
n^{-1},p^{-1}\right) ^{p}\left( n^{\prime -1},p^{\prime -1}\right) \\
= & \left( n,p\right)\left( n^{\prime },p^{\prime }\right) \left(
n^{-1},p^{-1}\right) \left( ^{\phi \left( p\right) }n^{\prime
-1},pp^{\prime
-1}p^{-1}\right) \\
= & \left( nn^{\prime }n^{-1},pp^{\prime }p^{-1}\right) \left(
^{\partial
_{1}\left( n\right) }n^{\prime -1},pp^{\prime -1}p^{-1}\right) \\
= & \left( nn^{\prime }n^{-1}\left( ^{\partial _{1}\left( n\right)
}n^{\prime -1}\right) ,pp^{\prime }p^{-1}pp^{\prime -1}p^{-1}\right) \\
= & \left( nn^{\prime }n^{-1}\left( ^{\partial _{1}\left( n\right)
}n^{\prime -1}\right) ,1\right) \\
= & \left( \partial _{2}\left\{ n,n^{\prime }\right\} ,1\right) \\
= & \partial _{2}^{\ast }\left\{ n,n^{\prime }\right\} \\
= & \partial _{2}^{\ast }\left\{ \left( n,p\right) ,\left( n^{\prime
},p^{\prime }\right) \right\} .%
\end{array}%
\end{equation*}

\textbf{PL2:}%
\begin{equation*}
\begin{array}{rll}
\left\{ \partial _{2}^{\ast }h,\partial _{2}^{\ast }h^{\prime
}\right\} & = & \left\{ \left( \partial _{2}h,1\right) ,\left(
\partial _{2}h^{\prime
},1\right) \right\} \\
& = & \left\{ \partial _{2}h,\partial _{2}h^{\prime }\right\} \\
& = & \left[ h,h^{\prime }\right] .%
\end{array}%
\end{equation*}

The rest of axioms of $2$-crossed module is given in appendix.

(ii)
\begin{equation*}
\left( \phi ^{\prime \prime },\phi ^{\prime },\phi \right) :\left\{
\partial _{2}^{-1}(Ker\partial _{1}),\phi ^{\ast }(N),P,\partial
_{2}^{\ast },\partial _{1}^{\ast }\right\} \rightarrow \left\{
H,N,Q,\partial _{2},\partial _{1}\right\}
\end{equation*}%
or diagrammatically,%
\begin{equation*}
\xymatrix { \partial _{2}^{-1}(Ker\partial _{1})
\ar[d]_-{\partial_{2}^*} \ar[r]^-{\phi^{\prime\prime}} & H
\ar[d]^-{\partial_2} &
\\ \ \phi^*(N) \ar[d]_-{\partial_{1}^*} \ar[r]^-{\phi^{\prime}} & N
\ar[d]^-{\partial_1}& \\ \ P \ar[r]_-\phi & Q & }
\end{equation*}%
is a morphism of $2$-crossed modules. (See appendix.)\

Suppose that%
\begin{equation*}
\left( f_{2},f_{1},\phi \right) :\left\{ B_{2},B_{1},P,\partial
_{2}^{\prime },\partial _{1}^{\prime }\right\} \rightarrow \left\{
H,N,Q,\partial _{2},\partial _{1}\right\}
\end{equation*}%
is any $2$-crossed module morphism
\begin{equation*}
\xymatrix { B_2 \ar[r]^{\partial^{'}_{2}} \ar[d]_{f_2} & B_1
\ar[d]_{f_1} \ar[r]^{\partial^{'}_{1}} & P \ar[d]_\phi & \\ \  H
\ar[r]_{\partial_{2}} & N \ar[r]_{\partial_{1}} & Q. & \\ }
\end{equation*}%
Then we will show that there is a unique $2$-crossed module morphism%
\begin{equation*}
\left( f_{2}^{\ast },f_{1}^{\ast },id_{P}\right) :\left\{
B_{2},B_{1},P,\partial _{2}^{\prime },\partial _{1}^{\prime
}\right\} \rightarrow \left\{ \partial _{2}^{-1}(Ker\partial
_{1}),\phi ^{\ast }(N),P,\partial _{2}^{\ast },\partial _{1}^{\ast
}\right\}
\end{equation*}%
\begin{equation*}
\xymatrix { B_2 \ar[r]^-{\partial^{'}_{2}} \ar[d]_-{f_2^*} & B_1
\ar[d]_-{f_1^*} \ar[r]^-{\partial^{'}_{1}} & P \ar@{=}[d]_-{id_P} & \\
\  \partial _{2}^{-1}(Ker\partial _{1}) \ar[r]_-{\partial _{2}^*} &
\phi^*(N) \ar[r]_-{\partial _{1}^*} & P & \\ }
\end{equation*}%
where $f_{2}^{\ast }(b_{2})=$ $f_{2}(b_{2})$ and $f_{1}^{\ast
}(b_{1})=(f_{1}(b_{1}),\partial _{1}^{\prime }(b_{1}))$ which is an
element in $\phi ^{\ast }(N).$ First let us check that $\left(
f_{2}^{\ast
},f_{1}^{\ast },id_{P}\right) $ is a $2$-crossed module morphism. \ For $%
b_{1},b_{1}^{\prime }\in B_{1},b_{2}\in B_{2},p\in P$%
\begin{equation*}
\begin{array}{rcl}
^{id_{P}\left( p\right) }f_{2}^{\ast }\left( b_{2}\right)  & = &
^{p}f_{2}\left( b_{2}\right)  \\
& = & ^{\phi \left( p\right) }f_{2}\left( b_{2}\right)  \\
& = & f_{2}\left( ^{p}b_{2}\right)  \\
& = & f_{2}^{\ast }\left( ^{p}b_{2}\right) .%
\end{array}%
\end{equation*}%
Similarly $^{id_{P}\left( p\right) }f_{1}^{\ast }\left( b_{1}\right) =$ $%
f_{1}^{\ast }\left( ^{p}b_{1}\right) ,$ also above diagram is
commutative and
\begin{equation*}
\begin{array}{cclcc}
\left\{ -,-\right\} \left( f_{1}^{\ast }\times f_{1}^{\ast }\right)
(b_{1},b_{1}^{\prime }) & = & \left\{ -,-\right\} \left( f_{1}^{\ast
}(b_{1}),f_{1}^{\ast }(b_{1}^{\prime })\right)  &  &  \\
& = & \left\{ -,-\right\} (\left( f_{1}(b_{1}),\partial _{1}^{\prime
}(b_{1})\right) ,\left( f_{1}(b_{1}^{\prime }),\partial _{1}^{\prime
}(b_{1}^{\prime }\right) ) &  &  \\
& = & \left\{ f_{1}(b_{1}),f_{1}(b_{1}^{\prime })\right\}  &  &  \\
& = & \left\{ -,-\right\} \left( f_{1}\times f_{1}\right) \left(
b_{1},b_{1}^{\prime }\right)  &  &  \\
& = & f_{2}\left\{ -,-\right\} \left( b_{1},b_{1}^{\prime }\right)  &  &  \\
& = & f_{2}\left\{ b_{1},b_{1}^{\prime }\right\}  &  &  \\
& = & f_{2}^{\ast }\left\{ b_{1},b_{1}^{\prime }\right\}  &  &  \\
& = & f_{2}^{\ast }\left\{ -,-\right\} \left( b_{1},b_{1}^{\prime
}\right) &  &
\end{array}%
\end{equation*}%
for $b_{1},b_{1}^{\prime } \in B_{1}.$ \\
Furthermore; the verification of the following equations are immediate.%
\begin{equation*}
\begin{array}{ccc}
\phi ^{\prime \prime }f_{2}^{\ast }=f_{2} & \text{and} & \phi
^{\prime
}f_{1}^{\ast }=f_{1}.%
\end{array}%
\end{equation*}
\end{proof}

Thus we get a functor
\begin{equation*}
\mathbf{\phi }^{\ast }\mathbf{:}\ \mathsf{X}_{2}\mathsf{Mod}\mathbf{/}%
\mathsf{Q}\mathbf{\ \ \rightarrow }\ \mathsf{X}_{2}\mathsf{Mod}\mathbf{/}%
\mathsf{P}
\end{equation*}%
which gives our pullback $2$-crossed module and its left adjoint functor
\begin{equation*}
\mathbf{\phi }_{\ast }\mathbf{:}\ \mathsf{X_{2}Mod}\mathbf{/}P\mathbf{\ \
\rightarrow }\ \mathsf{X_{2}Mod}\mathbf{/}Q
\end{equation*}
gives an induced $2$-crossed module which will be mentioned in
section \ref{i2xm}.
\subsection{Example of Pullback Two-Crossed Modules}

Given $2$-crossed module $%
\begin{array}{c}
\left\{ \left\{ 1\right\} ,G,Q,1,i\right\}%
\end{array}%
$ where $i$ is an inclusion of a normal subgroup and a morphism $\phi
:P\rightarrow Q$\ of groups, the pullback $2$-crossed module is%
\begin{equation*}
\begin{array}{ccl}
\phi ^{\ast }\left\{ \left\{ 1\right\} ,G,Q,1,i\right\} & = & \left\{
\left\{ 1\right\} ,\phi ^{\ast }(G),P,\partial _{2}^{\ast },\partial
_{1}^{\ast }\right\} \\
& = & \left\{ \left\{ 1\right\} ,\phi ^{^{-1}}(G),P,\partial _{2}^{\ast
},\partial _{1}^{\ast }\right\}%
\end{array}%
\end{equation*}%
as%
\begin{equation*}
\begin{array}{ccc}
\phi ^{\ast }(G) & = & \left\{ (g,p)\mid \phi (p)=i(g),g\in G,p\in P\right\}
\\
& \cong & \left\{ p\in P\mid \phi (p)=g\right\} =\phi ^{-1}(G)\unlhd P.%
\end{array}%
\end{equation*}%
The pullback diagram is
\begin{equation*}
\xymatrix { \{1\} \ar[d]_-{\partial_{2}^{*}=1} \ar@{=}[r] & \{1\} \ar[d] & \\
\ \phi^{-1}(G) \ar[d]_-{\partial_{1}^{*}} \ar[r] & G \ar[d]^-i & \\ \ P
\ar[r]_-\phi & Q. & }
\end{equation*}

Particularly if $G=\{1\}$, then%
\begin{equation*}
\phi ^{\ast }(\{1\})\cong \left\{ p\in P\mid \phi (p)=1\right\} =\ker \phi
\cong \{1\}
\end{equation*}%
and so $%
\begin{array}{c}
\{\{1\},\{1\},P,\partial _{2}^{\ast },\partial _{1}^{\ast }\}%
\end{array}%
$ is a pullback $2$-crossed module. \

Also if $\phi $ is an isomorphism and $G=Q,$ then $\phi ^{\ast }(Q)=Q\times
P.$

Similarly when we consider examples given in Section $1$, the following
diagrams are pullbacks.%
\begin{equation*}
\xymatrix { \{1\} \ar[d]_{\partial_{2}^{*}} \ar@{=}[r] & \{1\} \ar[d] & \\ \
\phi^{*}(Q) \ar[d]_{\partial_{1}^{*}} \ar[r] & Q \ar[d] & \\ \ Aut (P)
\ar[r]_{Aut \phi} & Aut (Q) & }\xymatrix { \{1\} \ar[d]_{\partial_{2}^{*}}
\ar@{=}[r] & \{1\} \ar[d] & \\ \ N \times Ker\phi \ar[d]_{\partial_{1}^{*}}
\ar[r] & N \ar[d]^0 & \\ \ P \ar[r]_\phi & Q& }
\end{equation*}

\section{Induced Two-Crossed Modules}\label{i2xm}

In this section we will construct induced $2$-crossed modules by extending the
discussion about induced crossed modules in subsection \ref{go}.

\begin{definition}
For any $2$-crossed module $L\overset{\partial}{\rightarrow }M\overset{%
\alpha}{\rightarrow }P$ and any pre-crossed module morphism $(M\overset{\alpha }{\rightarrow }P)\overset{(\phi^{\prime},\phi) }{\longrightarrow } (N\overset{\beta }{%
\rightarrow }Q),$
the $2$-crossed module induced by $\theta=(\phi^{\prime},\phi)$ from $\left\{ L,M,P,\partial,\alpha \right\}$
 is given by\\

(i) a $2$-crossed module $\left\{ \theta _{\ast }\left( L\right),N ,Q,\partial_{*},\beta \right\}
$

(ii) given any morphism of $2$-crossed modules%
\begin{equation*}
(f,\phi^{\prime},\phi ):\left\{ L,M,P,\partial,\alpha\right\}
\longrightarrow \{B,N,Q,\partial^{\prime },\beta \}
\end{equation*}%
then there is a unique morphism of $2$-crossed modules $(f_{\ast
},id_N,id_{Q})$ that makes the following diagram commute:
\begin{equation*}
\xymatrix { (L, M, P, \partial, \alpha)
\ar[drr]^{(\phi^{''},\phi^{'},\phi)} \ar[d]_{(f,\phi^{'},\phi)} & &
&
\\ \ (B, N, Q, \partial^{'}, \beta) & & (\theta_*(L),
N, Q,
\partial_{*}, \beta) \ar@{-->}[ll]^-{(f_{\ast},id_N,id_Q)}& }
\end{equation*}%
or more simply as%
\[
\xymatrix @R=20pt@C=25pt { & & & B  \ar[dd]^{\partial^\prime} & \\
\ L  \ar[dd]_{\partial} \ar[rr]_{\phi^{\prime \prime}} \ar@/^/[urrr]^{f} &  & {\theta_*}(L) \ar@{-->}[ur]_{f_\ast} \ar[dd]^{\partial_\ast} &  & \\
\ & & & ( N \rightarrow  Q) & \\
\ (M \rightarrow P)  \ar[rr]_\theta \ar@/^/[urrr]^{(\phi^{'},\phi)}
& & (N \rightarrow Q). \ar@{-->}[ur]_{(id_N,id_Q)}  & & }
\]

\end{definition}
The $2$-crossed module  $\left\{ \phi _{\ast }\left( L\right)
,N ,Q,\partial_{*},\beta\right\}$ is called the induced $2$-crossed module of $\left\{ L,M,P,\partial,\alpha \right\}$
along $\theta=(\phi^{\prime},\phi)$.

\subsection{Construction of Induced $2$-Crossed Modules} \label{go1}

The idea of the construction of induced $2$-crossed modules is the same as that of induced crossed module; we put all the data
in a free group and divide this group by all relations which we need to have
the properties we want.

Given a morphism
\begin{equation*}
\theta =(\phi ^{\prime },\phi ):(M\overset{\alpha }{\rightarrow }P)\longrightarrow (N\overset{\beta }{\rightarrow }Q)
\end{equation*}

\noindent of pre-crossed modules, we can define the $2$-crossed module induced by $%
\theta $, then we put things together. Thus we have the commutative diagram%
\begin{equation*}
\xymatrix{ L \ar[r] \ar[d]_{\partial} \ar[dr]^{\phi ^{\prime }{\partial}} & ^{Q}L \ar[d] & \\
\ {(M \rightarrow P)} \ar[r]_{(\phi ^{\prime },\phi )} & {(N \rightarrow Q)} & }
\end{equation*}%
where the complex
$
L\overset{\partial}{\rightarrow }M\overset{\alpha
}{\rightarrow }P
$
is a $2$-crossed module, $^{Q}L$  is a free group generated by L and the first morphism of the complex
\begin{equation*}
^{Q}L\rightarrow N\rightarrow Q
\end{equation*}
of $Q$-groups is induced
by $\phi ^{\prime }\partial.$

We consider the free product
\begin{equation*}
^{Q}L\ast \langle N\times N\rangle
\end{equation*}%
where $\langle N\times N\rangle $ is the free group generated by the set $%
N\times N.$ One can see that the action of $Q$ on $N$ induces an action on $%
\langle N\times N\rangle $ by
\begin{equation*}
^{q}(n_{1},n_{2})=(^{q}n_{1},^{q}n_{2})
\end{equation*}
for $n_{1},n_{2} \in N$ and $q\in Q$. Thus we get a morphism
\begin{equation*}
\bar{\partial}: \ ^{Q}L\ast \langle N\times N\rangle \rightarrow N
\end{equation*}%
given by
\begin{equation*}
\bar{\partial}((q,l)(n_{1},n_{2}))=\text{ }^{q}\phi \partial (l)\langle
n_{1},n_{2}\rangle
\end{equation*}%
where $\langle n_{1},n_{2}\rangle =n_{1}n_{2}n_{1}^{-1}\left(\text{
}^{\beta (n_{1})} n_{2}^{-1}\right)$
for $n_{1},n_{2}\in N.$ It is clear that the Peiffer lifting is given by $%
\left\{ n_{1},n_{2}\right\} =(n_{1},n_{2})$ for $n_{1},n_{2}\in N.$

Thus we get the free $(N\rightarrow Q)$-$2$-crossed module generated by $L$
denoted $\{{\theta_{\ast }(L),N,Q,\partial_*,\beta}\}$ by dividing the group $^{Q}L\ast \langle N\times N\rangle $ by $S$ generated by all elements of the following relations
\begin{equation*}
\left.
\begin{array}{c}
\left\{ \bar{\partial}((q,l)(n_{1},n_{2})),\bar{\partial}((q^{\prime },l^{\prime })(n_{1}^{\prime },n_{2}^{\prime
}))\right\} =\left[
(q,l)(n_{1},n_{2}),(q^{\prime },l^{\prime })(n_{1}^{\prime },n_{2}^{\prime })%
\right]  \\
\left\{ n_{0},n_{1}n_{2}\right\} =\ ^{n_{0}n_{1}n_{0}^{-1}}\left\{
n_{0},n_{2}\right\} \left\{ n_{0},n_{1}\right\}  \\
\left\{ n_{0}n_{1},n_{2}\right\} =\left\{
n_{0},n_{1}n_{2}^{-1}n_{1}\right\}
\left\{ n_{0},n_{1}\right\}  \\
\left\{ \bar{\partial}((q,l)(n_{1},n_{2})),n\right\} =(q,l)(n_{1},n_{2})%
\text{ }^{n}\left( (q,l)(n_{1},n_{2})\right) ^{-1} \\
\left\{ n,\bar{\partial}((q,l)(n_{1},n_{2}))\right\} =\ ^{n}\left(
(q,l)(n_{1},n_{2})\right) \text{ }^{\beta (n)}\left(
(q,l)(n_{1},n_{2})\right) ^{-1}%
\end{array}%
\right\}\ \ \ \ (\ref{go1}.1)
\end{equation*}%
for $ l,l^{\prime } \in L,q,q^{\prime } \in Q$ and
$n,n_{0},n_{1},n_{2},n_{1}^{\prime},n_{2}^{\prime} \in N.$
\\

We note that relations PL$1$ and PL$5$ are
given as follows:

\begin{equation*}
\partial _{\ast }\left\{ n_{1,}n_{2}\right\} =\partial _{\ast
}(n_{1,}n_{2})=\langle {n_{1},n_{2}\rangle }
\end{equation*}%
and%
\begin{equation*}
^{q}\left\{ n_{1,}n_{2}\right\} =\text{ }^{q}(n_{1,}n_{2})=(\text{ }%
^{q}n_{1,}\text{ }^{q}n_{2})=\left\{ ^{q}n_{1,}\text{ }^{q}n_{2}\right\} .
\end{equation*}

To get the $2$-crossed module induced by $\theta $, we add the relations
\begin{equation*}
\left.
\begin{array}{c}
\left\{ m_{1},m_{2}\right\} =\left\{ \phi ^{\prime }(m_{1}),\phi ^{\prime
}(m_{2})\right\}  \\
(\phi (p),l)=(1,^{p}l)%
\end{array}
\right\} \ \ \ \ \ \ \ (\ref{go1}.2)
\end{equation*}%
for $m_{1},m_{2} \in M,l \in L$ and $p \in P.$

Thus we have the following commutative diagram%
\[
\xymatrix{ L \ar[r]^-{\phi ^{\prime \prime}} \ar[d]_{\partial}  & ({^{Q}L\ast \langle N\times N\rangle})/S \ar[d]^{\partial_{\ast }} & \\
\ M \ar[r]^{\phi ^{\prime}} \ar[d]_{\alpha} & N \ar[d]^{\beta} & \\
\ P \ar[r]^{\phi} & Q & }
\]
where $S$ is the normal subgroup generated by the relation (\ref{go1}.1) and ${\phi ^{\prime \prime}}$ is defined as ${\phi ^{\prime \prime}}(l)=(1,l).$
Also, since
we have relation ($\ref{go1}.2),\ \theta=(\phi ^{\prime },\phi )$ is a morphism of $2$-crossed modules  and
\begin{equation*}
\beta \partial _{\ast }((q,l)\left( n_{1,}n_{2}\right) )=\beta (^{q}({\phi
^{\prime }\partial (l))}\langle {n_{1},n_{2}\rangle })=\text{ }^{q}({\phi
\alpha \partial (l))\beta ({n}_{1}n_{2}n_{1}^{-1\beta (n_{1})}n_{2}^{-1})=1}
\end{equation*}
so $\{{\theta_{\ast }(L),N,Q,\partial_*,\beta}\}$ is a complex of groups.
Thus $\{{\theta_{\ast }(L),N,Q,\partial_*,\beta}\}$ is an induced $(N \rightarrow
Q)$-$2$-crossed modules.

Then we obtain analogous result in dimension $3$ which extends
proposition \ref{go2} as follows:

\begin{theorem}
\label{ug4} Let $L\overset{\partial }{\rightarrow }M\overset{\alpha }{%
\rightarrow }P$ be a $2$-crossed module and let
\begin{equation*}
\theta =(\phi ^{\prime },\phi ):(M\overset{\alpha }{\rightarrow }%
P)\longrightarrow (N\overset{\beta }{\rightarrow }Q)
\end{equation*}%
be a pre-crossed module morphism and let $\theta _{\ast }(L)$ be the
quotient of the free product\ ${^{Q}L}$ $\ast $ ${\langle N\times
N\rangle }$ by $S$ \ where
${^{Q}L}$ is generated by the set $Q\times L$ with defining relations%
\begin{equation*}
\begin{array}{c}
(q,l_{1})(q,l_{2})=(q,l_{1}l_{2}) \\
(\phi (p),l)=(1,^{p}l)%
\end{array}%
\end{equation*}%
and the free group ${\langle N\times N\rangle }$ is generated by the
set $N\times N$ and $S$ is the normal subgroup generated by the
relations PL$2$, PL$3$ and PL$4$ in the definition of a $2$-crossed
module. Then
\begin{equation*}
\theta _{\ast }(L)\overset{\partial _{\ast }}{\rightarrow }N\overset{%
\alpha }{\rightarrow }Q
\end{equation*}%
is an induced $(N\rightarrow Q)$-$2$-crossed module where $\partial _{\ast }$
is given by
\begin{equation*}
\partial _{\ast }\ ((q,l)\left( n_{1,}n_{2}\right) )\ =\ (^{q}{\phi ^{\prime
}\partial (l)})\langle {n_{1},n_{2}\rangle }
\end{equation*}%
together with the Peiffer lifting $\left\{ n_{1,}n_{2}\right\}
=(n_{1,}n_{2}).$
\end{theorem}
\begin{proof}
We have only to check the universal property. For any morphism of $2$%
-crossed modules
\begin{equation*}
(f,\phi ^{\prime },\phi):\left\{ L,M,P,\partial ,\alpha \right\}
\rightarrow \left\{ B,N,Q,\partial ^{\prime },\beta \right\}
\end{equation*}%
there is a unique morphism of $(N\rightarrow Q)$-$2$-crossed modules
\[
(f_{\ast },id_{N},id_{Q}):\left\{ \theta _{\ast }(L),N,Q,\partial _{\ast
},\beta \right\} \rightarrow \left\{ B,N,Q,\partial ^{\prime },\beta
\right\}
\]
given by $f_{\ast }((q,l)\left( n_{1,}n_{2}\right) )=\ %
^{q}f(l)\left\{ n_{1,}n_{2}\right\} :$

\begin{equation*}
\begin{array}{rcl}
\partial ^{\prime }f_{\ast }((q,l)(n_{1},n_{2})) & = & \partial ^{\prime
}(^{q}f(l)\left\{ n_{1},n_{2}\right\} ) \\
& = & ^{q}\left( \partial ^{\prime }f(l)\right) \partial ^{\prime
}\left\{
n_{1},n_{2}\right\}  \\
& = & ^{q}\left( \phi ^{\prime }\partial (l)\right) n_{1}n_{2}n_{1}^{-1}%
\text{ }^{\beta (n_{1})}n_{2}^{-1} \\
& = & ^{q}\left( \phi ^{\prime }\partial (l)\right) \langle
n_{1},n_{2}\rangle  \\
& = & \partial _{\ast }((q,l)(n_{1},n_{2})).%
\end{array}%
\end{equation*}
Furthermore the verification of the following equations are
immediate:
\begin{equation*}
f_{\ast }\left( ^{q^{\prime }}\left( (q,l)(n_{1},n_{2})\right) \right) =%
\text{ }^{q}f_{\ast }\left( (q,l)(n_{1},n_{2})\right) \qquad
\text{and\qquad }f_{\ast }\left\{ n_{1},n_{2}\right\} =\left\{
n_{1},n_{2}\right\}
\end{equation*}%
for $l \in L, n_{1},n_{2} \in N$ and $q,q^\prime \in Q$ and $f_{\ast
}\phi ^{\prime \prime }=f$ is satisfied as required.
\end{proof}

In the case when $(\phi ^{\prime },\phi ):(M,P,\alpha )\rightarrow (N,Q,\beta
)$ is an epimorphism or a monomorphism of pre-crossed modules, we get proposition \ref{ug1} and proposition \ref{ug2} in dimension $3$ in terms of proposition \ref{gs1} and proposition \ref{gs2} given by Brown for crossed modules.

\begin{proposition}
\label{ug1}If $L\overset{\partial }{\rightarrow }M\overset{\alpha }{\rightarrow }P$ is a $%
2$-crossed module and $\theta =(\phi ^{\prime },\phi ):(M,P,\alpha
)\rightarrow
(N,Q,\beta )$ is an epimorphism of  pre-crossed modules Ker$\phi =K$ and Ker$%
\phi ^{\prime }=T$, then
\begin{equation*}
\theta _{\ast }(L)\cong L/[K,L]
\end{equation*}%
where $[K,L]$ denotes the subgroup of $L$ generated by $\left\{
^{k}ll^{-1}\mid k\in K,l\in L\right\} $ for all $k\in K,l\in L.$
\end{proposition}

\begin{proof}
As $\theta =(\phi ^{\prime },\phi ):(M,P,\alpha )\rightarrow
(N,Q,\beta )$ is
an epimorphism of pre-crossed modules, $Q\cong P/K$ and  $N\cong M/T.$ Since $K$ acts on trivially on $L/[K,L],$ $Q\cong
P/K$ acts on $L/[K,L]$ by  $^{q}(l[K,L])=$ $^{pK}(l[K,L])=$ $\left(
^{p}l\right)
[K,L].$%
\begin{equation*}
L/[K,L]\overset{\partial _{\ast }}{\rightarrow }N\overset{\beta }{%
\rightarrow }Q
\end{equation*}%
is a $2$-crossed module where $\partial _{\ast }(l[K,L])=\partial
(l)T,\ \beta (mT)=\alpha (m)K.$ As
\begin{equation*}
\beta \partial _{\ast }(l[K,L])=\beta (\partial (l)T)=\alpha
(\partial (l))K=K,
\end{equation*}%
$L/[K,L]\overset{\partial _{\ast }}{\rightarrow }N\overset{\beta }{%
\rightarrow }Q$ is a complex of groups.

The Peiffer lifting%
\begin{equation*}
N\times N\rightarrow L/[K,L]
\end{equation*}%
is given by $\left\{ mT,m^{\prime }T\right\} =\left\{ m,m^{\prime
}\right\} [K,L].$

\textbf{PL1:}%
\begin{equation*}
\begin{array}{rcl}
\partial _{\ast }\left\{ mT,m^{\prime }T\right\}  & = & \partial _{\ast
}(\left\{ m,m^{\prime }\right\} [K,L]) \\
& = & (\partial \left\{ m,m^{\prime }\right\} )T \\
& = & (mm^{\prime }m^{-1}\text{ }^{\alpha (m)}m^{\prime -1})T \\
& = & mTm^{\prime }T(m^{-1}T)\text{ }(^{\alpha (m)}m^{\prime -1})T \\
& = & mTm^{\prime }T(mT)^{-1}\text{ }^{\alpha (m)K}(m^{\prime }T)^{-1} \\
& = & mTm^{\prime }T(mT)^{-1}\text{ }^{\beta (mT)}(m^{\prime }T)^{-1}%
\end{array}%
\end{equation*}

\textbf{PL2:}%
\begin{equation*}
\begin{array}{rcl}
\left\{ \partial _{\ast }(l[K,L]),\partial _{\ast }(l^{\prime
}[K,L])\right\}  & = & \left\{ \partial (l)T,\partial (l^{\prime
})T\right\}
\\
& = & \left\{ \partial (l),\partial (l^{\prime })\right\} [K,L] \\
& = & [l,l^{\prime }][K,L] \\
& = & \left[ l[K,L],l^{\prime }[K,L]\right]
\end{array}%
\end{equation*}

\noindent
for $m,m^\prime \in M$ and $l,l^\prime \in L$. The rest of
axioms of the $2$-crossed module is given in appendix.
\begin{equation*}
(\phi ^{\prime \prime },\phi ^{\prime },\phi ):\left\{
L,M,P,\partial ,\alpha \right\} \rightarrow \left\{
L/[K,L],N,Q,\partial _{\ast },\beta \right\}
\end{equation*}%
or diagrammatically,%
\begin{equation*}
\xymatrix { L \ar[r]^{\phi^{''}} \ar[d]_\partial & L/[K,L]
\ar[d]^{\partial _\ast} & \\ \ M \ar[r]^-{\phi^\prime}
\ar[d]_{\alpha}& N \ar[d]^{\beta} & \\ \ P \ar[r]^\phi & Q & }
\end{equation*}%
is a morphism of $2$-crossed modules. (See appendix.)

Suppose that%
\begin{equation*}
(f,\phi ^{\prime },\phi ):\left\{ L,M,P,\partial ,\alpha \right\}
\rightarrow \left\{ B,N,Q,\partial ^{\prime },\beta \right\}
\end{equation*}%
is any $2$-crossed module morphism. Then we will show that there is
a unique $2$-crossed module morphism
\begin{equation*}
(f_{\ast },id_{N},id_{Q}):\left\{ L/[K,L],N,Q,\partial _{\ast
},\beta \right\} \rightarrow \left\{ B,N,Q,\partial ^{\prime },\beta
\right\}
\end{equation*}%
\begin{equation*}
\xymatrix { L/[K,L] \ar[r]^-{\partial^{\prime}_{2}} \ar[d]_-{f_{\ast }} &
N \ar[d]_-{id_{N}} \ar[r]^-{\beta} & Q \ar[d]_-{id_{Q}} & \\ \ B
\ar[r]_-{\partial ^{\prime }} & N \ar[r]_-{\beta} & Q & \\ }
\end{equation*}%
where $f_{\ast }(l[K,L])=f(l)$ such that $(f_{\ast},id_N,id_Q)(\phi^{\prime\prime},\phi^{\prime},\phi)=(f,\phi^{\prime},\phi)$. Since
\begin{equation*}
f(^{k}ll^{-1})=f(^{k}l)f(l^{-1})=\text{ }^{\phi
(k)}f(l)f(l)^{-1}=1_{B},
\end{equation*}%
$f_{\ast }$ is well-defined.

First let us check that $(f_{\ast },id_{N},id_{Q})$ is a $2$-crossed
module
morphism, for $l[K,L]\in L/[K,L]$ and $q\in Q,$%
\begin{equation*}
\begin{array}{rcl}
f_{\ast }(^{q}(l[K,L])) & = & f_{\ast }(^{pK}(l[K,L])) \\
& = & f_{\ast }(^{p}l[K,L]) \\
& = & f(^{p}l) \\
& = & ^{\phi (p)}f(l) \\
& = & ^{pK}f_{\ast }(l[K,L]) \\
& = & ^{q}f_{\ast }(l[K,L]),%
\end{array}%
\end{equation*}%
$\partial ^{\prime }f_{\ast }=\partial _{\ast }$ and
\begin{equation*}
\begin{array}{rcl}
f_{\ast }\left\{ mT,m^{\prime }T\right\}  & = & f_{\ast }(\left\{
m,m^{\prime }\right\} [K,L]) \\
& = & f\left\{ m,m^{\prime }\right\}  \\
& = & \left\{ mT,m^{\prime }T\right\} .%
\end{array}%
\end{equation*}%
Thus $(f_{\ast },id_{N},id_{Q})$ is a morphism of $2$-crossed modules.
Furthermore the following equation is verified%
\begin{equation*}
f_{\ast }\phi ^{\prime \prime }=f.
\end{equation*}%
So given any morphism of $2$-crossed modules
\begin{equation*}
(f,\phi ^{\prime },\phi ):\left\{ L,M,P,\partial ,\alpha \right\}
\rightarrow \left\{ B,N,Q,\partial ^{\prime },\beta \right\} ,
\end{equation*}%
then there is a unique $(f_{\ast },id_{N},id_{Q})$ $2$-crossed
module
morphism that makes the following diagram commutes:%
\begin{equation*}
\xymatrix { \{L, M, P, \partial, \alpha\}
\ar[drr]^{(\phi^{''},\phi^{'},\phi)} \ar[d]_{(f,\phi^\prime,\phi)} &
& &
\\ \ \{B,N,Q, \partial^{'}, \beta\} & & \{L/[K,L],
N, Q,
\partial_{\ast}, \beta\} \ar@{-->}[ll]^-{(f_{\ast},id_N,id_Q)}& }
\end{equation*}%
or more simply as%
\begin{equation*}
\xymatrix @R=20pt@C=25pt { & & & B  \ar[dd]^{{\partial}^\prime} & \\
\ L  \ar[dd]_{\partial} \ar[rr]_{\phi^{\prime \prime}} \ar@/^/[urrr]^{f} &  & L/[K,L] \ar@{-->}[ur]_{f_\ast} \ar[dd]^{\partial_\ast} &  & \\
\ & & & N  \ar[dd]^{\beta} & \\
\ M \ar[dd]_{\alpha} \ar[rr]_{\phi^\prime} \ar@/^/[urrr]^{\phi^\prime} &  & N \ar@{=}[ur]_{id_N} \ar[dd]^{\beta} & & \\
\  & & & Q  & \\
\  P \ar@/^/[urrr]^\phi \ar[rr]_\phi &  & Q \ar@{=}[ur]_{id_Q}.& & }
\end{equation*}
\end{proof}

\begin{proposition}
\label{ug2}Let $\theta =(\phi ^{\prime },\phi ):(M,P,\alpha )\rightarrow
(N,Q,\beta )$ be a pre-crossed module morphism where $\phi $ is a
monomorphism and $\left\{ L,M,P,\partial ,\alpha \right\} $ be a $2$-crossed
module, let $T$ be a left transversal of $\phi (P)$ in $Q,$ and $%
_{T}L\ast \langle N\times N\rangle $ be the free product of $_{T}L$ and $%
\langle N\times N\rangle $ where $_{T}L$ is the free product of groups $%
(t\in T)$ each isomorphic with $L$ by an isomorphism $l\mapsto $ $_{t}l$ $%
(l\in L)$ and $\langle N\times N\rangle $ is the free group generated by the
set $N\times N. $ If $q\in Q$ acts on $_{T}L\ast \langle N\times N\rangle $
by the rule $^{q}(_{t}l)=$ $^{u}(_{p}l)$ $\ $where $p\in P,u\in T$ and $%
qt=u\phi (p),$ then $\theta _{\ast }(L)=$ $(_{T}L\ast \langle N\times N\rangle )/S$ and $%
\left\{ \theta _{\ast }(L),N,Q,\bar{\partial},\beta \right\} $ is a $2$%
-crossed module with the Peiffer lifting $\left\{ n_{1},n_{2}\right\}
=(n_{1},n_{2})$ where $S$ is the normal closure in $_{T}L\ast \langle
N\times N\rangle $ of the elements
\begin{equation*}
\begin{array}{c}
\left\{ \bar{\partial}(_{t}l(n_{1},n_{2})),\bar{\partial}(_{t^{\prime
}}l^{\prime }(n_{1}^{\prime },n_{2}^{\prime }))\right\} =\left[
_{t}l(n_{1},n_{2}),\ _{t^{\prime }}l^{\prime }(n_{1}^{\prime },n_{2}^{\prime })%
\right] \\
\left\{ n_{0},n_{1}n_{2}\right\} =\ ^{n_{0}n_{1}n_{0}^{-1}}\left\{
n_{0},n_{2}\right\} \left\{ n_{0},n_{1}\right\} \\
\left\{ n_{0}n_{1},n_{2}\right\} =\left\{ n_{0},n_{1}n_{2}^{-1}n_{1}\right\}
\left\{ n_{0},n_{1}\right\} \\
\left\{ \bar{\partial}(_{t}l(n_{1},n_{2})),n\right\} =_{t}l(n_{1},n_{2})%
\text{ }^{n}\left( _{t}l(n_{1},n_{2})\right) ^{-1} \\
\left\{ n,\bar{\partial}(_{t}l(n_{1},n_{2}))\right\} =\ ^{n}\left(
_{t}l(n_{1},n_{2})\right) \text{ }^{\beta (n)}\left(
_{t}l(n_{1},n_{2})\right) ^{-1} \\
\left\{ m_{1},m_{2}\right\} =\left\{ \phi ^{\prime }(m_{1}),\phi ^{\prime
}(m_{2})\right\}%
\end{array}%
\end{equation*}%
for $l,l^{\prime }\in L,n,n_{0},n_{1},n_{2}\in N,$ $t,t^{\prime }\in T$ and $%
m_{1},m_{2}\in M.$
\end{proposition}

Now consider an arbitrary push-out square

\begin{equation*}
\xymatrix{ \{L_{0},M_{0},P_{0},\partial_{0},\alpha_{0}\} \ar[d]
\ar[r] & \{L_{1},M_{1},P_{1},\partial_{1},\alpha_{1}\} \ar[d] &
(\ref{go1}.3)\\ \
\{L_{2},M_{2},P_{2},\partial_{2},\alpha_{2}\}\ar[r] &
\{L,M,P,\partial,\alpha\} & }
\end{equation*}%
of $2$-crossed modules. In order to describe $\left\{ L,M,P,\partial ,\alpha
\right\} $, we first note that $(M\overset{\alpha }{\rightarrow }P)$ is the
push-out of the pre-crossed module morphisms $(M_{1},P_{1},\alpha
_{1})\leftarrow (M_{0},P_{0},\alpha _{0})\rightarrow (M_{2},P_{2},\alpha
_{2}).$ (This is because the functor
\begin{equation*}
\left\{ L,M,P,\partial ,\alpha \right\} \mapsto (M,P,\alpha )
\end{equation*}%
from $2$-crossed modules to pre-crossed modules has a right adjoint $%
(M,P,\partial )\mapsto \left\{ \langle M,M\rangle ,M,P,i,\alpha
\right\}.) $ The pre-crossed module morphisms
\[
\theta_{i}=(\phi _{i}^{\prime },\phi
_{i}):(M_{i},P_{i},\alpha _{i})\\ \rightarrow (M,P,\alpha )
\]
$(i=0,1,2)$ in $%
(\ref{go1}.3) $ and can be used to form induced $(M\overset{\alpha }{%
\rightarrow }P)$-$2$-crossed modules $B_{i}=\left( \theta
_{i}\right) _{\ast
}(L_{i}).$ $L$ is the push-out in \textsf{X}$_{2}$\textsf{Mod}${/(M,P)}$ of the resulting $(M%
\overset{\alpha }{\rightarrow }P)$-$2$-crossed module morphisms
\begin{equation*}
\left( B_{1}\rightarrow M\rightarrow P\right) \overset{(\beta
_{1},id_{M},id_{P})}{\longleftarrow }\left( B_{0}\rightarrow M\rightarrow
P\right) \overset{(\beta _{2},id_{M},id_{P})}{\longrightarrow }\left(
B_{2}\rightarrow M\rightarrow P\right)
\end{equation*}%
and can be described as follows.

\begin{proposition}
Let $B_{i}$ be a $(M\overset{\alpha }{\rightarrow }P)$-$2$-crossed module
for $i=0,1,2$ and let $L$ be the push-out in \textsf{X}$_{2}$\textsf{Mod}${/(M,P)}$ of $(M\overset{%
\alpha }{\rightarrow }P)$-$2$-crossed module morphisms

\begin{equation*}
 B_{1} \overset{\beta
_{1}}{\longleftarrow } B_{0} \overset{\beta _{2}}{\longrightarrow }
B_{2}.
\end{equation*}%

\noindent Let $B$ be the push-out of $\beta _{1}$ and $\beta _{2}$ in the category of
groups, equipped with the induced morphism $B\overset{\partial }{\rightarrow
}(M\overset{\alpha }{\rightarrow }P$) such that $\alpha\beta=1$
and the induced action of $P$ on $B$ $.$ Then $L=B/S,$ where $S$ is the
normal closure in $B$ of the elements
\begin{equation*}
\begin{array}{c}
\left\{ \partial \left( b\right) ,\partial \left( b^{\prime }\right)
\right\} \left[ b,b^{\prime }\right] ^{-1} \\
\left\{ m,m^{\prime }m^{\prime \prime }\right\} \left\{ m,m^{\prime
}\right\} ^{-1}\left( ^{mm^{\prime }m^{-1}}\left\{ m,m^{\prime \prime
}\right\} \right) ^{-1} \\
\left\{ mm^{\prime },m^{\prime \prime }\right\} \left( ^{\alpha
\left( m\right) }\left\{ m^{\prime },m^{\prime \prime }\right\}
\right)
^{-1}\left\{ m,m^{\prime }m^{\prime \prime }m^{\prime -1}\right\} ^{-1} \\
\left\{ \partial \left( b\right) ,m\right\} \left( ^{m}b^{-1}\right)
^{-1}b^{-1} \\
\left\{ m,\partial \left( b\right) \right\} \left( ^{\alpha \left( m\right)
}b^{-1}\right) ^{-1}\left( ^{m}b\right) ^{-1} \\
^{p}\left\{ m,m^{\prime }\right\} \left\{ ^{p}m,^{p}m^{\prime }\right\} ^{-1}%
\end{array}%
\end{equation*}%
for $b,b^{\prime }\in B,m,m^{\prime },m^{\prime \prime }\in M$ and
$p\in P.$
\end{proposition}

In the case when $\left\{ L_{2},M_{2},P_{2},\partial _{2},\alpha
_{2}\right\} $ is the trivial $2$-crossed module $\left\{
1,1,1,id,id\right\} $ the push-out $\left\{ L,M,P,\partial
_{2},\partial _{1}\right\} $ in (\ref{go1}.3) is the
cokernel of the morphism
\begin{equation*}
\left\{ L_{0},M_{0},P_{0},\partial _{0},\alpha _{0}\right\}
\rightarrow \left\{ L_{1},M_{1},P_{1},\partial _{1},\alpha
_{1}\right\}.
\end{equation*}

\section {Appendix}

\textbf{The proof of Theorem} \textbf{\ref{ug3}} \\

\bigskip

\ \textbf{PL3: } \newline

$%
\begin{array}{cl}
a) & ^{\left( n,p\right) \left( n^{\prime },p^{\prime }\right) \left(
n,p\right) ^{-1}}\left\{ \left( n,p\right) ,\left( n^{\prime \prime
},p^{\prime \prime }\right) \right\} \left\{ \left( n,p\right) ,\left(
n^{\prime },p^{\prime }\right) \right\}  \\
= & ^{\left( nn^{\prime }n^{-1},pp^{\prime }p^{-1}\right) }\left\{
n,n^{\prime \prime }\right\} \left\{ n,n^{\prime }\right\}  \\
= & ^{nn^{\prime }n^{-1}}\left\{ n,n^{\prime \prime }\right\} \left\{
n,n^{\prime }\right\} \text{ } \\
= & \left\{ n,n^{\prime }n^{\prime \prime }\right\}  \\
= & \left\{ \left( n,p\right) ,\left( n^{\prime }n^{\prime \prime
},p^{\prime }p^{\prime \prime }\right) \right\}  \\
= & \left\{ \left( n,p\right) ,\left( n^{\prime },p^{\prime }\right) \left(
n^{\prime \prime },p^{\prime \prime }\right) \right\}
\end{array}%
\bigskip $

$%
\begin{array}{ll}
b) & \left\{ \left( n,p\right) ,\left( n^{\prime },p^{\prime }\right) \left(
n^{\prime \prime },p^{\prime \prime }\right) \left( n^{\prime },p^{\prime
}\right) ^{-1}\right\} \left( ^{\partial _{1}^{\ast }\left( n,p\right)
}\left\{ \left( n^{\prime },p^{\prime }\right) ,\left( n^{\prime \prime
},p^{\prime \prime }\right) \right\} \right)  \\
= & \left\{ \left( n,p\right) ,\left( n^{\prime }n^{\prime \prime }n^{\prime
-1},p^{\prime }p^{\prime \prime }p^{\prime -1}\right) \right\} ^{p}\left\{
\left( n^{\prime },p^{\prime }\right) ,\left( n^{\prime \prime },p^{\prime
\prime }\right) \right\} \text{ } \\
= & \left\{ n,n^{\prime }n^{\prime \prime }n^{\prime -1}\right\} ^{p}\left\{
n^{\prime },n^{\prime \prime }\right\} \text{ } \\
= & \left\{ n,n^{\prime }n^{\prime \prime }n^{\prime -1}\right\} ^{\phi
\left( p\right) }\left\{ n^{\prime },n^{\prime \prime }\right\} \text{ } \\
= & \left\{ n,n^{\prime }n^{\prime \prime }n^{\prime -1}\right\} ^{\partial
_{1}\left( n\right) }\left\{ n^{\prime },n^{\prime \prime }\right\} \text{ }
\\
= & \left\{ nn^{\prime },n^{\prime \prime }\right\} \text{ } \\
= & \left\{ \left( nn^{\prime },pp^{\prime }\right) ,\left( n^{\prime \prime
},p^{\prime \prime }\right) \right\}  \\
= & \left\{ \left( n,p\right) \left( n^{\prime },p^{\prime }\right) ,\left(
n^{\prime \prime },p^{\prime \prime }\right) \right\}
\end{array}%
\bigskip $

for $(n,p),(n^{\prime },p^{\prime }),(n^{\prime \prime },p^{\prime \prime
})\in \phi ^{\ast }(N)$.

\bigskip

\textbf{PL4: }

\bigskip $%
\begin{array}{rcl}
\left\{ \partial _{2}^{\ast }h,\left( n,p\right) \right\} \left\{ \left(
n,p\right) ,\partial _{2}^{\ast }h\right\}  & = & \left\{ \left( \partial
_{2}h,1\right) ,\left( n,p\right) \right\} \left\{ \left( n,p\right) ,\left(
\partial _{2}h,1\right) \right\} \text{ } \\
& = & \left\{ \partial _{2}h,n\right\} \left\{ n,\partial _{2}h\right\}
\text{ } \\
& = & h\text{ }^{\partial _{1}\left( n\right) }h^{-1}\text{ } \\
& = & h\text{ }^{\phi (p)}h^{-1}\text{ } \\
& = & h\text{ }^{p}h^{-1}\text{ } \\
& = & h\text{ }^{\partial _{1}^{\ast }\left( n,p\right) }h^{-1}%
\end{array}%
\bigskip $

for $(n,p)\in \phi ^{\ast }(N)$ and $h\in \partial _{2}^{-1}(Ker\partial _{1}).
$

$\bigskip $

\textbf{PL5:}

$%
\begin{array}{rcl}
\left\{ ^{p^{\prime \prime }}\left( n,p\right) ,^{p^{\prime \prime }}\left(
n^{\prime },p^{\prime }\right) \right\}  & = & \left\{ \left( ^{\phi \left(
p^{\prime \prime }\right) }n,p^{\prime \prime }p\left( p^{\prime \prime
}\right) ^{-1}\right) ,\left( ^{\phi \left( p^{\prime \prime }\right)
}n^{\prime },p^{\prime \prime }p^{\prime }\left( p^{\prime \prime }\right)
^{-1}\right) \right\} \text{ } \\
& = & \left\{ ^{\phi \left( p^{\prime \prime }\right) }n,^{\phi \left(
p^{\prime \prime }\right) }n^{\prime }\right\} \text{ } \\
& = & ^{\phi \left( p^{\prime \prime }\right) }\left\{ n,n^{\prime }\right\}
\text{ } \\
& = & ^{p^{\prime \prime }}\left\{ n,n^{\prime }\right\} \text{ } \\
& = & ^{p^{\prime \prime }}\left\{ \left( n,p\right) ,\left( n^{\prime
},p^{\prime }\right) \right\}
\end{array}%
\bigskip $

for $\left( n,p\right) ,\left( n^{\prime },p^{\prime }\right) \in \phi
^{\ast }(N)$ and $p^{\prime \prime }\in P.$ \bigskip

$%
\begin{array}{rcllrll}
\phi ^{\prime \prime }\left( ^{p}h\right)  & = & ^{p}h & \text{and} & \phi
^{\prime }({^{p}(n,p^{\prime })}) & = & \phi ^{\prime }\left( ^{\phi \left(
p\right) }n,pp^{\prime }p^{-1}\right)  \\
& = & ^{\phi \left( p\right) }h &  &  & = & ^{\phi \left( p\right) }n \\
& = & ^{\phi \left( p\right) }\phi ^{\prime \prime }\left( h\right)  &  &  &
= & ^{\phi \left( p\right) }\phi ^{\prime }\left( n,p^{\prime }\right) %
\end{array}%
\bigskip $

$%
\begin{array}{rlllrll}
\phi ^{\prime }\left( \partial _{2}^{\ast }h\right)  & = & \phi ^{\prime
}\left( \partial _{2}h,1\right)  & \text{and} & \partial _{1}\left( \phi
^{\prime }\left( n,p^{\prime }\right) \right)  & = & \partial _{1}\left(
n\right)  \\
& = & \partial _{2}\left( h\right)  &  &  & = & \phi \left( p^{\prime
}\right)  \\
& = & \partial _{2}\left( \phi ^{\prime \prime }h\right)  &  &  & = & \phi
\left( \partial _{1}^{\ast }\left( n,p^{\prime }\right) \right) %
\end{array}%
\bigskip $

$%
\begin{array}{rcl}
\left\{ -,-\right\} \left( \phi ^{\prime }\times \phi ^{\prime }\right)
\left( (n,p),\left( n^{\prime },p^{\prime }\right) \right)  & = & \left\{
-,-\right\} \left( \phi ^{\prime }(n,p),\phi ^{\prime }(n^{\prime
},p^{\prime })\right)  \\
& = & \left\{ -,-\right\} (n,n^{\prime }) \\
& = & \left\{ n,n^{\prime }\right\}  \\
& = & \phi ^{\prime \prime }\left( \left\{ n,n^{\prime }\right\} \right)  \\
& = & \phi ^{\prime \prime }\left( \left\{ (n,p),\left( n^{\prime
},p^{\prime }\right) \right\} \right)  \\
& = & \phi ^{\prime \prime }\left\{ -,-\right\} \left( (n,p),\left(
n^{\prime },p^{\prime }\right) \right)
\end{array}%
$ \bigskip

for $h\in \partial _{2}^{-1}(Ker\partial _{1}),(n,p),(n^{\prime },p^{\prime
}),(n,p^{\prime })\in \phi ^{\ast }(N)$ and $p\in P$.

\bigskip

\textbf{The proof of proposition} \textbf{\ref{ug1}:} \\

\textbf{PL3:}%
\begin{equation*}
\begin{array}{rcl}
\left\{ mT,m^{\prime }Tm^{\prime \prime }T\right\}  & = & \left\{
m,m^{\prime }m^{\prime \prime }\right\} [K,L] \\
& = & (^{mm^{\prime }m^{-1}}\left\{ m,m^{\prime \prime }\right\}
\left\{
m,m^{\prime }\right\} )[K,L] \\
& = & ^{mm^{\prime }m^{-1}}\left\{ m,m^{\prime \prime }\right\}
[K,L]\left\{
m,m^{\prime }\right\} [K,L] \\
& = & ^{(mm^{\prime }m^{-1})T}(\left\{ m,m^{\prime \prime }\right\}
[K,L])\left\{ m,m^{\prime }\right\} [K,L] \\
& = & ^{mTm^{\prime }Tm^{-1}T}\left\{ m,m^{\prime \prime }\right\}
[K,L]\left\{ m,m^{\prime }\right\} [K,L] \\
& = & ^{mTm^{\prime }T(mT)^{-1}}\left\{ mT,m^{\prime \prime
}T\right\} \left\{ mT,m^{\prime }T\right\}
\end{array}%
\end{equation*}%
\begin{equation*}
\begin{array}{rcl}
\left\{ mTm^{\prime }T,m^{\prime \prime }T\right\}  & = & \left\{
mm^{\prime
},m^{\prime \prime }\right\} [K,L] \\
& = & (\left\{ m,m^{\prime }m^{\prime \prime }m^{\prime -1}\right\} \text{ }%
^{\alpha (m)}\left\{ m^{\prime },m^{\prime \prime }\right\} )[K,L] \\
& = & \left\{ m,m^{\prime }m^{\prime \prime }m^{\prime -1}\right\} [K,L]%
\text{ }(^{\alpha (m)K}\left\{ m^{\prime },m^{\prime \prime
}\right\} )[K,L]
\\
& = & \left\{ mT,m^{\prime }Tm^{\prime \prime }T(m^{\prime
}T)^{-1}\right\} \text{ }^{\beta (mT)}\left\{ m^{\prime }T,m^{\prime
\prime }T\right\}
\end{array}%
\end{equation*}
for $m,m^{\prime },m^{\prime \prime }\in M.$

\bigskip

\textbf{PL4:}%
\begin{equation*}
\begin{array}{rcl}
a)\left\{ \partial _{\ast }(l[K,L]),mT\right\}  & = & \left\{
\partial
(l)T,mT\right\}  \\
& = & \left\{ \partial (l),m\right\} [K,L] \\
& = & (l\text{ }^{m}l^{-1})[K,L] \\
& = & l[K,L]\text{ }^{m}l^{-1}[K,L] \\
& = & l[K,L]\text{ }^{mT}(l[K,L])^{-1}%
\end{array}%
\end{equation*}%
\begin{equation*}
\begin{array}{rcl}
b)\left\{ mT,\partial _{\ast }(l[K,L])\right\}  & = & \left\{
mT,\partial
(l)T\right\}  \\
& = & \left\{ m,\partial (l)\right\} [K,L] \\
& = & (^{m}l\text{ }^{\alpha (m)}l^{-1})[K,L] \\
& = & ^{m}l[K,L]\text{ }^{\alpha (m)}l^{-1}[K,L] \\
& = & ^{mT}(l[K,L])\text{ }^{\alpha (m)K}(l^{-1}[K,L]) \\
& = & ^{mT}(l[K,L])\text{ }^{\beta (mT)}(l[K,L])^{-1}%
\end{array}%
\end{equation*}
for $l \in L$ and $m\in M.$

\bigskip

\textbf{PL5:}%
\begin{equation*}
\begin{array}{rcl}
^{pK}\left\{ mT,m^{\prime }T\right\}  & = & ^{pK}(\left\{
m,m^{\prime
}\right\} [K,L]) \\
& = & ^{p}\left\{ m,m^{\prime }\right\} [K,L] \\
& = & \left\{ ^{p}m,^{p}m^{\prime }\right\} [K,L] \\
& = & \left\{ ^{p}mT,^{p}m^{\prime }T\right\}  \\
& = & \left\{ ^{pK}(mT),^{pK}(m^{\prime }T)\right\}
\end{array}%
\end{equation*}%
for $m,m^{\prime }\in M$ and $p\in P.$

\begin{equation*}
\begin{array}{rcl}
\phi ^{\prime \prime }(^{p}l) & = & ^{p}l[K,L] \\
& = & ^{pK}(l[K,L]) \\
& = & ^{\phi (p)}\phi ^{\prime \prime }(l)%
\end{array}%
\qquad \text{and\qquad }%
\begin{array}{rcl}
\phi ^{\prime }(^{p}m) & = & ^{p}mT \\
& = & ^{pK}(mT) \\
& = & ^{\phi (p)}\phi ^{\prime }(m)%
\end{array}%
\end{equation*}%
\begin{equation*}
\begin{array}{rcl}
\partial _{\ast }\phi ^{\prime \prime }(l) & = & \partial _{\ast }(l[K,L])
\\
& = & \partial (l)T \\
& = & \phi ^{\prime }\partial (l)%
\end{array}%
\qquad \text{and}\qquad
\begin{array}{rcl}
\beta \phi ^{\prime }(m) & = & \beta (mT) \\
& = & \alpha (m)K \\
& = & \phi \alpha (m)%
\end{array}%
\end{equation*}%
and similarly $\phi ^{\prime \prime }\left\{ -,-\right\} =\left\{
-,-\right\} (\phi ^{\prime }\times \phi ^{\prime })$ for $l\in
L,m\in M$ and $p\in P.$

\noindent U. Ege Arslan, Z. Arvas\.{i} and G. Onarl\i \newline
Department of Mathematics-Computer \newline
Eski\c{s}ehir Osmangazi University \newline
26480 Eski\c{s}ehir/Turkey\newline
e-mails: \{uege,zarvasi, gonarli\} @ogu.edu.tr

\end{document}